\documentclass[a4paper,12pt]{article}

\usepackage{cmap}

\usepackage{amsmath,amsthm,amssymb}
\usepackage[T2A]{fontenc}
\usepackage[english]{babel}
\usepackage[cp1251]{inputenc}
\usepackage{amscd}

\usepackage{graphics}

\usepackage{caption}
\usepackage{subcaption}

\usepackage{hyperref}

  \newskip\prethm \prethm3.0pt plus1.3pt minus.4pt
  \newskip\posthm \posthm2.7pt plus1.4pt minus.3pt
  \newtheoremstyle{STATEMENT}%
       {\prethm}{\posthm}{\itshape}{\parindent}{\scshape}%  {\bfseries}
       {.}{.6em plus.2em minus.1em}{}
  \newtheoremstyle{EXPLANATION}%
       {\prethm}{\posthm}{}{\parindent}{\scshape}%  {\bfseries}
       {.}{.6em plus.2em minus.1em}{}
\theoremstyle{STATEMENT}

\newtheorem {theorem}{Theorem}
\newtheorem {lemma}{Lemma}
\newtheorem    {proposition} { Proposition}

\theoremstyle{EXPLANATION}

\newtheorem {definition} {Definition}
\newtheorem   {remark}{Remark}
\newtheorem    {example} {Example}

\begin{document}

\title{Liouville classification of integrable geodesic flows on a projective plane in a potential field}
\author{Antonov~E.\,I.\thanks{M. V. Lomonosov Moscow State University, Moscow, Russia.} \\ antonov.zhenya@hotmail.com \and Kozlov~I.\,K.\thanks{No Affiliation, Moscow, Russia.} \\  ikozlov90@gmail.com}
\date{\today}

\maketitle

\begin{abstract}  A Liouville classification of a natural Hamiltonian system on the projective plane with a rotation metric and a linear integral is obtained. All Fomenko--Zieschang invariants (i.e., labeled molecules) of the system are calculated. \end{abstract} 

\begin{center}
\end{center}

\begin{flushleft}
{\bf  UDC 517.938.5}
\end{flushleft}

\section{Introduction}

\noindent
In this paper, we study topological properties of an integrable geodesic flow on $\mathbb{RP}^2$ that is a quotient by involution of the geodesic flow on $S^2$, considered in the paper of E. O. Kantonistova~\cite{Kant}. Namely, we calculate all the Fomenko-Zieschang invariants (that is, labeled molecules) of this system (see Theorem~\ref{Th:main1}).

The results are based on the theory of topological classification of integrable Hamiltonian systems, created by A.\,T.~Fomenko and his school (see \cite{BolFom}). The foundations of this theory were laid in \cite{FomenkoMorseTh86, FomenkoIsoTop86, FomenkoZieschang87, FomenkoFunk88}.

For more details on the Liouville classification (that is, on the calculation of the Fomenko--Zieschang invariants) for geodesic flows, see \cite{BolFom}. In particular, the Fomenko--Zieschang invariants for a linearly integrable geodesic flow on $\mathbb{RP}^2$ without potential were calculated by V.\,S.~Matveev (see \cite[Volume 2, Section 3.4]{BolFom}). In \cite {Kant} E.\,O.~ Kantonistova calculated all labeled molecules of integrable geodesic flows on surfaces of revolution homeomorphic to the sphere $S^2$ with an invariant potential and a linear integral. D.\,S.~Timonina in \cite{Tim} described all possible labeled molecules for geodesic flows on $\mathbb{RP}^2$ that are a quotient of the systems from \cite{Kant} by an involution. We briefly describe these systems on $S^2$ and $\mathbb{RP}^2$.

\subsection{Linearly integrable geodesic flows on $S^2$}

Consider a Riemannian manifold of revolution $M=S^2$ with natural coordinates $(r;\varphi), r \in (0;L), \varphi \in \mathbb{R}/2\pi\mathbb{Z}$, in which the rotation metric is written as \[\mathrm{d}s^2=\mathrm{d}r^2+f^2(r)\mathrm{d}\varphi^2.\] In a neighbourhood of the poles consider local coordinates \[x=f(r)\cos{\varphi},\qquad y=f(r)\sin{\varphi}.\] The function $f(r)$ satisfies the conditions of the following lemma, therefore the metric at the poles is $\mathrm{d}s^2=\mathrm{d}x^2+\mathrm{d}y^2$.

\begin{lemma}[A.~Besse, \cite{besse}] \label{L:Besse} A metric on a manifold of revolution $M$ and a function $V(r)$ invariant under rotations are smooth at the poles (that is at the points $r=0$ and $r = L$) if there exist $F=F(r)$ and $W=W(r)$ defined on the whole number line such that $F|_{\left[0;L\right]}=f, W|_{\left[0;L\right]}=V$ and the following conditions are satisfied:
\begin{enumerate}
  \item $F(-r)=-F(r)=F(2L-r)$,  that is the functions $F(r)$ and $F(L+r)$ are odd (or, equivalently, the function $F(r)$
  is periodic with a period $2L$ and odd) and $F'(0)=1, F'(L)=-1$;
  
   \item $W(-r)=W(r)=W(2L-r)$, that is the functions $W(r)$ and $W(L+r)$ are even  (or, equivalently, the function  $W(r)$ is  periodic with a period $2L$ and even).
\end{enumerate}
\end{lemma}

Consider an integrable Hamiltonian system on the cotangent bundle  $T^*S^2$ with Hamiltonian\begin{equation} \label{Eq:Ham}  H = \begin{cases} \displaystyle \frac{p^2_{r}}{2}+\frac{p^2_{\varphi}}{2f^2(r)}+V(r),  \qquad & \text{outside the poles  (i.e. $r \not =0, L$),} \\ \displaystyle \frac{p^2_{x}+p^2_{y}}{2}+V(0), \qquad & \text{at the poles (i.e. $r = 0$ or $L$)} \end{cases}  \end{equation} 
and first integral  \begin{equation} \label{Eq:Int}  K= \begin{cases} p_{\varphi}, \qquad  &\text{outside the poles,} \\ xp_{y}-yp_{x}, \qquad & \text{near the poles.}  \end{cases} \end{equation}

\begin{remark}[\cite{Kant}] \label{Rem:K_Zero} At the poles (that is at $x=y =0$) $K=0$. For  $K=0$  the system does not have critical points of rank $1$ on nonsingular isoenergy surfaces $Q^3_h$. Moreover, the system on $T^*S^2$ has exactly two points of rank $0$: these are  $(0,N)$ and $(0,S)$, where $N$ and $S$ are the poles of the sphere. \end{remark}

Consider a Liouville torus of the system $H = h, K=k$  that does not contain points of the poles. Then it consists of the points $(p_{r}, k, r, \varphi)$, $\varphi \in \mathbb{R}/2\pi\mathbb{Z}$, satisfying the equation \begin{equation} \label{Eq:Pr_LiouvTorus} p_{r}=\pm \frac{1}{f{(r)}}\sqrt{U_{h}(r)-k^2},\end{equation} where \begin{equation} \label{Eq:Uhr} U_{h}(r) =2f^2(r)(h-V(r)).\end{equation}

\begin{remark}  In \cite{Kant} it is explained how to construct a labeled molecule of the system by the graph of the function  $U_h (r)$ (see also \cite{BolFom} and \cite{Tim}). The molecule is always symmetric about the level $k=0$ and has the form $W - W$. Informally speaking, the Liouville foliation for the ``half of the molecule'' $W$ is obtained if we take the ``volumetric mountain relief'' of the function  $U_h(r)$, stratify it into the level sets $U_h = \operatorname{const}$ and then multiply it by the circle $S^1$. \end{remark}

\begin{theorem}[E.\,O.~Kantonistova  \cite{Kant}] Consider a system (that is, a geodesic flow with a linear integral and with a potential invariant under rotations) on a manifold of revolution $M\approx S^2$, given by a pair of functions $(V(r);f(r))$. Let $Q\subseteq Q^{3}_{h}$ be the connected component of a nonsingular isoenergy surface on which 
$K$ is a Bott function. Then:

\begin{enumerate}
  \item the molecule of the system is symmetric (without taking into account the orientation on the edges) with respect to the $h$  axis, and the orientation on the edges is set in the direction of increasing $k$. That is, the molecule has the form $W-W$, where the edge connecting the molecules is a one-parameter family of Liouville tori, and each  molecule $W$  is either a single atom $A$ or a tree. All non-terminal (that is saddle) vertices of the tree are atoms $V_{l}$ the terminal ones are of type $A$.  Moreover, for $k>0$  there is one  incoming edge for each atom $V_{l}$ and there are $l$ outgoing edges (for $k <0$, the picture is antisymmetric, that is, without taking into account the orientation on the edges the molecule is symmetric with respect to $h$, and the orientations on the pieces $W_{+} = W(k>0)$ and $W_{-} = W(k<0)$ are opposite).

\item

\begin{enumerate}
  \item  Labels on the edges of the type $A-V_{l}$ of the molecule are $r=0, \varepsilon=+1$.
  \item Labels on the edges of the type  $V_{s}-V_{l}$  where both saddle atoms are in the same half-plane ($k>0$ or $k<0$) are $r=\infty, \varepsilon=+1$.
  \item Labels on the central edge of the type $V_{l}-V_{l}$ (symmetric with respect to the level  $k=0$) are  $r=\infty, \varepsilon=-1$.
   
  \item \label{Item:SphereMolAA} If the molecule $W-W$ is of the type $A-A$, then the label $r$  is defined as follows: we cut the manifold $M^4$ along the surface $Q^3$  into two pieces  $M^4_{-}$ and $M^4_{+}$ (recall that  $Q^3$ is a connected component of $Q^3_{h}$).  The piece  $M^4_{-}$ that  corresponds to strictly lower energy values than $h$ can contain $2$, $1$ or $0$  singular points of rank zero. Then, respectively, $r=\frac{1}{2}$, $r=0$, or $r=\infty$. In all three cases,  $\varepsilon=+1$.
  
  \item If the molecule  $W-W$  is different from $A-A$, then it contains a single family that is obtained by dropping all the atoms $A$. The label $n$ in this case is equal to the number of singular points of rank zero on the manifold $M^4_{-}$ (see item~\ref{Item:SphereMolAA}).
\end{enumerate}
\end{enumerate}
\end{theorem}

\begin{remark} \label{Rem:TopQAA} Consider the molecule $A-A$. The label  $r=\frac{1}{2}, 0, \infty$  or $r = \frac{p}{q}$ (where $p<q$ or $q \geq 3$) if and only if the corresponding surface  $Q^3$ is homeomorphic to $\mathbb{RP}^3$,  $S^{3}$, $S^{1}\times S^{2}$ or the lens space $L_{q, p}$ respectively (see \cite[Volume 1, Proposition 4.3]{BolFom}\end{remark}

\subsection{Linearly integrable geodesic flows on $\mathbb{RP}^2$}

Consider a projective plane $\mathbb{RP}^2$ as a quotient of a sphere  $S^2$ by the involution $\eta$ that in the coordinates $r,\varphi$ is given by the formula: \begin{equation} \label{Eq:Invol} \eta (r, \varphi) = (L-r, \varphi + \pi).\end{equation} The poles switch places under the involution:  $\eta (S)=N,\eta(N)=S$.  We have:   $T^{*}\mathbb{RP}^2=T^{*}S^2/\eta ^{*}$, where \begin{equation} \label{Eq:CotangInvolution} \eta^{*}(p_{r},p_{\varphi},r,\varphi)=(-p_{r},p_{\varphi},L-r,\varphi+\pi).
 \end{equation} Further we assume that  $f(r)=f(L-r)$ and $V(r) = V(L-r)$ so that $f$  and  $V$ defined functions on  $\mathbb{RP}^2$.

 \begin{theorem}[D.\,S.~Timonina, \cite{Tim}] Let the system (that is, a geodesic flow with a linear integral on manifold of revolution with invariant potential) on the projective plane $\mathbb{RP}^2$ be given by a pair of functions $f(r)$ and $V(r)$ as described above. Then the molecule of the system corresponding to the connected component of the isoenergetic surface $Q^{3}_{\mathbb{RP}^2}$ is symmetric and has the form $W-W$,  where each molecule $W$  is either a single atom $A$ or a tree. All non-terminal (that is saddle) vertices of the tree are atoms $V_{l},V^{*}_{k}, A^{*}$ and the end vertices are of the type $A$. Labeled molecule are constructed by the graph of the function $U_{h}(r)=2f^2(r)(h-V(r))$. \end{theorem}

 \begin{remark}  Some properties of the function $U_h(r)$:
 
 \begin{enumerate}
 
 \item It follows from the symmetries of the functions  $f$ and $V$  that the function $U_{h}(r)$  is symmetric about the axis  $r = \frac{L}{2}$. In other words, the function $U_{h}(r-\frac{L}{2})$ is even.
 
 \item  At the poles  $U_h(0) = U_h (L) =0$. 
 
 \item  For $r \not = 0, L$  the function  $U_h(r)>0$  if and only if  $V(r) < h$. 
 
 \item Also note that for $K=0$ and $r \not =0, L$ the zeros of $U_h(r) =0 $  cannot be points of the local extremum of the function $U_h(r)$. Otherwise, at these points  $p_r=0$ and $\frac{\partial H}{\partial r} =0$, and these would be points of rank  $1$. But on nonsingular surfaces with $K=0$ there are no points of rank  $1$ (see Remark~\ref{Rem:K_Zero}).

 \end{enumerate}
 
  \end{remark}

It is easy to prove the following statement about the topology of $Q^3_{\mathbb{RP}^2}$.  Let $\pi: Q^3_h \to \mathbb{RP}^2$ be the projection of the isoenergy surface. Note that the points of $\pi(Q^3_h)$ are exactly the points where $V(r) \leq h$. The set $\pi(Q^3_h)$ is symmetric with respect to rotation, so each its connected component is either all $\mathbb{RP}^2$, or a ``cap'' at the pole, or a Mobius strip (a neighborhood of the ``equator''), or a ring.

 \begin{proposition}\label{A:IsoTop} Let $Q^3$ be the connected component of $Q^3_h$ of a nonsingular isoenergy surface for the system under consideration on $\mathbb{RP}^2$. There are three possible cases:

 \begin{enumerate}
 
 \item \label{I:ProjAll} If $\pi(Q^3) = \mathbb{RP}^2$, that is $V(r) \leq h$ on the whole $\mathbb{RP}^2$, then $Q^3 \approx L_{4, 1}$.
 
 \item If $\pi(Q^3)$ is a disk $D^2$ centered at the pole, then $Q^3 \approx S^3$.
 
 \item If $\pi(Q^3)$ is a ring $I^1 \times S^1$ or a Mobius strip  $\mathbb{M}^2$, then $Q^3 \approx S^1 \times S^2$. 

 \end{enumerate}
 
  \end{proposition}

 \begin{remark} In the cases of Proposition~\ref{A:IsoTop} the corresponding surface for the sphere  $Q^3_{S^2}$  is homeomorphic to $\mathbb{RP}^3$, $S^3$ and   $S^1 \times S^2$ respectively (see \cite{Kant}).   \end{remark}

\begin{proof}[Proof of Proposition~\ref{A:IsoTop}] If $\pi(Q^3) = \mathbb{RP}^2$,  then it is easy to see that  $Q^3$  is homeomorphic to the bundle of unit tangent vectors to  $\mathbb{RP}^2$.  It is well known that this space is homeomorphic to  $L_{4, 1}$(see, for example, \cite{GeigesLange}). If $\pi(Q^3) \approx D^2$ or $\pi(Q^3) \approx I^1 \times S^1$, then the answer is the same as for the corresponding case on the sphere $S^2$ (see \cite{Kant}).  It remains to consider the case of $\pi(Q^3) \approx \mathbb{M}^2$. In this case, $Q^3 \approx Q^3_{S^2}/\eta^*$,   where the involution $\eta^*$ is given by \eqref{Eq:CotangInvolution} and  $Q^3_{S^2}\approx S^1 \times S^2$. The involution $\eta^*$ acts on $Q^3_{S^2}$ as follows: 
it increases the coordinate $\varphi$ on $S^1$ by $\pi$,  and the map on the factor $S^2$  is isotopic to the identity map (in the coordinates $(p_r, p_{\varphi}, r)$  it is easy to see that it is isotopic to the rotation by the angle $\pi$ about the axis $p_{\varphi}$).  Therefore, $Q^3 \approx S^1 \times S^2/\eta^* \approx S^1 \times S^2$. Proposition~\ref{A:IsoTop} is proved. \end{proof}

\begin{remark}  For the molecules of the form  $A-A$  one can derive the label $r$ from the topology of  $Q^3$ (see Remark~\ref{Rem:TopQAA}). However, we will do the opposite: in Section~\ref{S:ProofTh1} in all three cases we first calculate the label $r$  for molecules of the form $A-A$.  And then make sure that the answer is consistent with  Proposition~\ref{A:IsoTop}. \end{remark}

 \section{Main results}
  
 For convenience, we divide the classification of labeled molecules for the considered system into two statements: for the case when the molecule has the form $A-A$, and for all other cases.
 
\begin{theorem}\label{Th:main1}
Consider a natural Hamiltonian system on  $\mathbb{RP}^2$  with Hamiltonian~\eqref{Eq:Ham} and  first integral~\eqref{Eq:Int}. Let $Q^3$ be a connected component of a nonsingular isoenergy surface $Q^3_{h}$. For all $A-A$ molecules the label $\varepsilon =1$ and the label $r$ depends on the projection of  $\pi(Q^3)$ on $\mathbb{RP}^2$ as follows.

\begin{enumerate}

\item \label{Item:RP2LensCase}  If $\pi(Q^3) = \mathbb{RP}^2$, then the label $r = \frac{1}{4}$.

\item \label{Item:RP2SphereCase}  If $\pi(Q^3) \approx D^2$, then the label $r =0$.

\item \label{Item:RP2S1S3Case} In all other cases (that is if $\pi (Q^3) \approx I^1 \times S^1$ or $\pi(Q^3) \approx \mathbb{M}^2$) the label $r =\infty$. 

\end{enumerate}

 \end {theorem}

 \begin{remark}   Let the projection of $Q^3$ on $\mathbb{RP}^2$ in the coordinates $(r, \varphi)$ have the form $\pi(Q^3) = [a, b] \times S^1$, where  $[a, b] \subseteq \left[0, \frac{L}{2} \right]$ (and we formally identify all points of the form $(0, \varphi)$). In other words, let  $[a, b] \subseteq \left[0, \frac{L}{2} \right]$be the connected component of the set of points  $r$  at which  $V(r) \leq h$. The molecule of the system on  $Q^3$ is of type $A - A$ if and only if the function  $U_h(r)$ on $[a, b]$ has exactly one positive local extremum (which will be a global maximum), see~\cite{Tim}. And the topology of $\pi(Q^3)$  depends on the form of the segment $[a, b]$ as follows.

 \begin{enumerate}
 
 \item If $a =0, b = \frac{L}{2}$, then $\pi(Q^3) = \mathbb{RP}^2$.
 
 \item If $a =0, b < \frac{L}{2}$, then $\pi(Q^3) \approx D^2$.
 
 \item If $a >0, b = \frac{L}{2}$, then $\pi(Q^3) \approx \mathbb{M}^2$.
 
 \item If $a >0, b < \frac{L}{2}$, then $\pi(Q^3) \approx I^1 \times S^1$.

 \end{enumerate} We also note that if the molecule has the form  $A-A$  and $\pi(Q^3) = \mathbb{RP}^2$ (that is $a =0, b = \frac{L}{2}$), then $U_h\left(\frac{L}{2}\right)$ is the maximum of the function. \end{remark}

\begin{example} The case when the potential  $V(r) =0$ was analyzed by V.\,S.~Matveev (see  \cite[Volume 2, Section 3.4]{BolFom}). In this case, if the molecule has the form  $A-A$, then $V(r) \leq h$ on the whole $\mathbb{RP}^2$. Therefore, the label $r = \frac{1}{4}$ and  $Q^3 \approx L_{4, 1}$.  \end{example}

\begin{theorem}\label{Th:main2} Let a molecule $W-W$ of the system on $Q^3$ be different from $A-A$. Then the labels on it are as follows.

\begin{enumerate}
  \item On  the edges of the type $A-V_{l}$ the label $\varepsilon =1$. If the atom  $A$ is central (see Definition~\ref{Def:AtomACent}), then the label $r=\frac{1}{2}$, otherwise the label $r=0$.

  \item On the edges between saddle atoms the labels $r=\infty$. If the edge is noncentral (that is $k>0$ or $k<0$), the label $\varepsilon=+1$. On the center edge the label $\varepsilon =-1$.

  \item If the molecule $W-W$  is different from  $A-A$, then it contains a single family obtained by discarding all atoms  $A$. The label  $n$ depends on the projection of  $\pi(Q^3)$  on $\mathbb{RP}^2$ as follows.

\begin{enumerate}

\item \label{Item:RP2LabelNRP2Case}  If $\pi(Q^3) = \mathbb{RP}^2$, then  the label $n=-2$.

\item \label{Item:RP2LabelNDiskMobiusCase}  If $\pi(Q^3) \approx D^2$ or $\pi(Q^3) \approx \mathbb{M}^2$, then the label $n=-1$.

\item \label{Item:RP2LabelNCylindCase} In the remaining case (if $\pi (Q^3) \approx I^1 \times S^1$) the label $n=0$.

\end{enumerate}

\end{enumerate}

 \end {theorem}

 \begin{remark} In \cite[Volume 2, Theorem 3.11]{BolFom}  in a result of V.\,S.~Matveev about the case of zero potential $V(r) =0$ the label $n$ is incorrect. If $V(r)=0$, then $\pi(Q^3) = \mathbb{RP}^2$ and the label is $n=-2$, not $-1$ (this is also confirmed by the Topalov formula, see Remark~\ref{Rem:Topalov}).   \end{remark}

\section{Construction of admissible bases} \label{S:AdmBasis}

In order to prove Theorems~\ref{Th:main1} and \ref{Th:main2} we explicitly construct admissible bases for all atoms and then calculate the labels by the rules from\cite{BolFom}. We construct admissible atoms separately for elliptic and saddle atoms.

\subsection{Case of elliptic atoms}

The only elliptic atom is the atom $A$. 

\begin{remark} For the atom $A$, we take an admissible basis $(\lambda, \mu)$ according to the following rule from  \cite{BolFom}:

\begin{enumerate}

\item $\lambda$-cycle is contractible;

\item $\mu$-cycle completes the cycle  $\lambda$  to a basis;

\item $\mu$-cycle is oriented by the direction of $\operatorname{sgrad} H$ on the critical circle;

\item $\lambda$-cycle is oriented so that the pair  $(\lambda, \mu)$ is positively oriented on the boundary torus  $\mathbb{T}^2$ of a solid torus.  We assume that a pair of vectors $u, v$ defines a positive orientation of the tangent space  $T_x \mathbb{T}^2$ if \begin{equation} \label{Eq:AtomAOrientCond} \omega \wedge\omega (\operatorname{grad} H, N, u, v)>0, \end{equation}  that is if the four-tuple $(\operatorname{grad} H, N, u, v)$, where $N$ is an outward normal vector to the solid torus, is positively oriented with respect to  $\omega \wedge \omega$. 

\end{enumerate}

\end{remark}

Recall (see \cite{BolFom}) that the $\lambda$-cycle for an atom $A$ is uniquely determined, and the  $\mu$--cycle is given up to an addition of $k \lambda$, where $k \in \mathbb{Z}$.

\begin{remark}  \label{Rem_Syst_Sym} The mapping \begin{equation} \label{Eq:SystSym} (p_r, p_{\varphi}, r, \varphi) \to (p_r, - p_{\varphi}, r, - \varphi),\end{equation} preserves the symplectic structure  $\omega = dp_r \wedge dr + d p_{\varphi} \wedge d \varphi$ on $T^* S^2$ with $H \to H, K \to -K$. We will construct admissible bases for $K<0$ from the admissible bases for $K>0$ using the mapping~\eqref{Eq:SystSym}.  In particular, for an atom  $A$ in the coordinates $(p_{r'}, p_{\varphi'}, r', \varphi') = (p_r, - p_{\varphi}, r, - \varphi)$ an  admissible bases  $(\lambda', \mu')$ is given by the same formulas as the basis $(\lambda, \mu)$. 
 \end{remark}

Further note that for $K \not =0$  the Liouville torus defined by \eqref{Eq:Pr_LiouvTorus}, (before factorization by the involution~\eqref{Eq:CotangInvolution}) is the product of two cycles:

\begin{enumerate}

\item a cycle $\alpha_{\varphi}$ of the form\begin{equation} \label{Eq:PhiCycle}
\Big\{p_{r}=const, \quad p_{\varphi}=const,\quad r=const, \quad \varphi \in\mathbb{R}/2\pi\mathbb{Z}\Big\},
\end{equation} where we assume that $\dot{\varphi} >0$;

\item  a cycle $\alpha_{r}$  of the form \begin{equation} \label{Eq:RCycle}
\Big\{p_{r}=\pm\frac{1}{f(r)}\sqrt{U_{h}(r)-k^{2}}, p_{\varphi}=const,r\in\left[r_{1},r_{2}\right]\subset\left[0,L\right], \varphi =const\Big\},
\end{equation} which, for definiteness, is clockwise (that is $\dot{r}>0$ for $p_r >0$).

\end{enumerate}

Each atom $A$ corresponds to a maximum of the function $U_h(r)$  for $K>0$ and to a minimum of the function $U_h(r)$ for $K<0$.  Let this value be reached at a point $r_{\mathrm{ext}}$.  For $\mathbb{RP}^2$ there will be two cases:

\begin{enumerate}

\item if $r_{\mathrm{ext}} \not = \frac{L}{2}$, then the involution~\eqref{Eq:CotangInvolution}  identifies two Liouville tori in a neighbourhood of the singular leaf of the atom $A$ (in Fig~\ref{SubFig:AnotCenter} the corresponding cycles $\alpha_r$ are shown);

\item if $r_{\mathrm{ext}}  = \frac{L}{2}$, then the involution~\eqref{Eq:CotangInvolution} takes a Liouville torus into itself in a neighbourhood of the singular leaf of the atom  $A$ (see Fig.~\ref{SubFig:ACenter})

\end{enumerate}

 \begin{figure}[h]
 \begin{subfigure}{0.5\textwidth}
  \centering
 \includegraphics{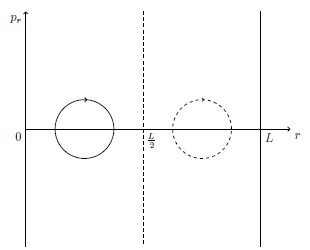}
 \caption{Noncentral atom}
\label{SubFig:AnotCenter}
\end{subfigure}
  \begin{subfigure}{0.5\textwidth}
   \centering
   \includegraphics{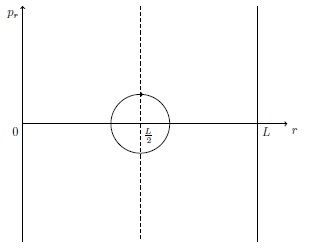}
 \caption{Central atom}
\label{SubFig:ACenter}
\end{subfigure}
 \caption{Section of an atom $A$}
\label{ris:image1}
\end{figure}

\begin{definition}\label{Def:AtomACent}  An atom $A$ corresponding to $r_{\mathrm{ext}} \not = \frac{L}{2} \displaystyle$ will be called  \textit{noncentral}, and an atom $A$ corresponding to $r_{\mathrm{ext}} = \frac{L}{2} \displaystyle$ will be called \textit{central}. \end{definition}

\begin{proposition} \label{A:AdmBas} For the system  on $\mathbb{RP}^2$  under consideration we can take the bases given in the Table~\ref{table:AdmBasisA}  as admissible bases $(\lambda, \mu)$ in a neighbourhood of an atom $A$ depending on the values of the integral $K$ and the type of the atom $A$.

\begin{table}[h]

\centering

{
\renewcommand{\arraystretch}{2}%

\begin{tabular} {c | c| c }
& Noncentral $A$ & Central $A$ \\ \hline
$K>0$ &  $\lambda = \alpha_r$, $\mu = \alpha_{\varphi}$ & $\lambda = \alpha_r$, $\mu = \frac{\alpha_r + \alpha_{\varphi}}{2} \displaystyle $  \\ \hline
$K<0$ &  $\lambda = \alpha_r$, $\mu = - \alpha_{\varphi}$ & $\lambda = \alpha_r$, $\mu = \frac{\alpha_r -\alpha_{\varphi}}{2}\displaystyle  $ \\
\end{tabular}

\caption{Admissible bases for an atom $A$}
\label{table:AdmBasisA}

}

\end{table}

\end{proposition}

\begin{proof}[Proof of Proposition~\ref{A:AdmBas}] In accordance with Remark ~\ref{Rem_Syst_Sym}   it suffices to prove the statement for $K>0$. First, consider noncentral atoms (see Fig.~\ref{SubFig:AnotCenter}). Obviously, the cycle $\lambda$ is contractible , and that  $(\lambda, \mu)$ is a basis. Further, the canonical symplectic structure on the cotangent bundle has the form $\omega = dp_r \wedge dr + dp_{\varphi} \wedge d\varphi$. Therefore, \begin{equation} \label{Eq:HamVectH} \operatorname{sgrad} H = \omega^{-1} dH = \left( - \frac{\partial H}{\partial r} , 0, p_r, \frac{p_{\varphi}}{f^2} \right). \end{equation} At singular points of rank $1$ the vector field $\operatorname{sgrad} H$  is proportional to  $\operatorname{sgrad} K = \left(0, 0, 0, 1 \right)$. Therefore, the orientation of the cycle  $\mu$  depends only on the sign of $p_{\varphi} = k$ and is chosen correctly.

It remains to verify that the orientation of $\lambda$ is chosen correctly, that is \eqref{Eq:AtomAOrientCond} holds. Denote by $v_{\lambda}$ and $v_{\mu}$ the tangent vectors to the cycles $\lambda$ and $\mu$. We assume that $p_r>0$. In this case, the four vectors are as follows \begin{gather*}  \operatorname{grad} H = \left( p_r, \frac{p_{\varphi}}{f^2}, \frac{\partial H}{\partial r}, 0 \right), \qquad  N = \left( \alpha, -1, -\alpha \beta, 0 \right) \\ v_{\lambda} = \left( \beta, 0, 1, 0 \right), \qquad v_{\mu} = \left( 0, 0, 0, 1\right),  \end{gather*} where \[ \alpha = \frac{k}{f^2} \frac{p_r}{p_r^2 + \left( \frac{\partial H}{\partial r} \right)^2}, \qquad \beta = - \frac{1}{p_r}  \frac{\partial H}{\partial r}.\] Note that the vector  $N$ is found from the conditions \[ (N, \operatorname{grad} H) = (N, v_{\lambda}) = (N, v_{\mu}) =0, \qquad N(p_{\varphi})<0.\]

The volume form $\omega \wedge \omega = - 2 d p_r \wedge dp_{\varphi} \wedge dr \wedge d \varphi$, so the four-tuple is positively oriented, since the determinant of the corresponding matrix is negative: \[\left| \begin{matrix} p_r & \frac{p_{\varphi}}{f^2} & \frac{\partial H}{\partial r} & 0 \\  \alpha & -1 & -\alpha \beta & 0 \\ \beta & 0 & 1 & 0 \\ 0 & 0 & 0 & 1 \end{matrix} \right| = - \frac{1}{p_r} \left( p_r^2 + \left(  \frac{\partial H}{\partial r}\right)^2 + \frac{k^2}{f^4} \right)<0,    \] because we took $p_r>0$. The condition~\eqref{Eq:AtomAOrientCond} is satisfied; therefore,  $(\lambda, \mu)$  is an admissible basis.

Now consider the central atoms (see Fig.~\ref{SubFig:ACenter}). It suffices to prove that  $(\lambda, \mu)$ is a basis on the Liouville torus (after factorization by the involution~\eqref{Eq:CotangInvolution}). The remaining statements (that $\lambda$ is contractible and that the orientations of $\lambda, \mu$  are chosen correctly) are proved in the same way as in the case  $r_{\mathrm{ext}} \not  = \frac{L}{2}$. Prior the factorization by the involution~\eqref{Eq:CotangInvolution}  we can take the same basis $(\lambda', \mu')$ as in the case $r_{\mathrm{ext}} \not  = \frac{L}{2}$. Introduce a parameter $\psi \in \mathbb{R}/ 2\pi \mathbb{Z}$ on a closed curve in Fig.~\ref{SubFig:ACenter}. Moreover, we can assume that the involution~\eqref{Eq:CotangInvolution} acts as $ (\varphi, \psi) \to (\varphi + \pi, \psi + \pi)$. In this case, before involution, the torus $\mathbb{T}^2$ was a quotient of $\mathbb{R}^2$ by the subgroup generated by the vectors $(2 \pi, 0)$ and $(0, 2\pi)$. After involution $(\pi, \pi)$ must be added to the generators of the subgroup. Hence $\lambda = \lambda'$ and $\mu = \frac{\lambda' + \mu'}{2}$ will really be the new basis in $\pi_1 (\mathbb{T}^2/ \mathbb{Z}^2)$.  Q.E.D.

Proposition~\ref{A:AdmBas} is proved.\end{proof}

\subsection{Case of saddle atoms}

There are two types of saddle atoms --- with and without stars.

\begin{remark} For a saddle atom without stars, we take admissible bases  $(\lambda, \mu_i)$ according to the following rule from \cite{BolFom}:

\begin{enumerate}

\item $\lambda$-cycle is a fiber the Seifert fibration (that is the cycle that is obtained from the critical circle by a continuous extension to neighboring tori);

\item $\mu_i$-cycles complement the cycle  $\lambda$  to a basis, and there must exist a global section of the $3$-atom passing simultaneously through all $\mu_i$.

\item $\lambda$-cycle is oriented by the direction of $\operatorname{sgrad} H$ on the critical circle;

\item $\mu_i$--cycle is oriented so that the pair $(\lambda, \mu_i)$  is positively oriented on the boundary torus  $\mathbb{T}^2$   (orientation on the torus  $\mathbb{T}^2$ is given by~\eqref{Eq:AtomAOrientCond}, as in the case of the atom $A$).

\end{enumerate}

\end{remark}

\begin{remark} \label{Rem:StarAtom}  For atoms with stars, in accordance with  \cite{BolFom}, we construct admissible bases for a double of an atom. In this case, the double of an atom is the corresponding atom for the sphere $S^2$. If the corresponding Liouville torus transforms into itself under involution~\eqref{Eq:CotangInvolution},  then the cycle $\hat{\mu}_i$ for the sphere  $S^2$  must be replaced by $\mu_i = \frac{\lambda + \hat{\mu}_i}{2}$. Otherwise, this cycle can be left unchanged: $\mu_i = \hat{\mu}_i$. \end{remark}

Recall (see  \cite{BolFom}),  that the $\lambda$-cycle for saddle atoms is uniquely defined, and the $\mu_i$-cycles are defined up to an addition of $k_i \lambda$, where $k_i \in \mathbb{Z}$ and $\sum_i k_i =0$.

\begin{definition}  We say that a saddle atom that goes into itself under the involution~\eqref{Eq:CotangInvolution} is \textit{central}. The remaining saddle atoms will be called  \textit{noncentral}. \end{definition}

\begin{proposition} \label{A:AdmBasSaddle} For the system  on $\mathbb{RP}^2$ under consideration  in a neighbourhood of the saddle atom, the $\lambda$ cycle of an admissible basis is  $\alpha_{\varphi}$ for $K>0$ and $-\alpha_{\varphi}$ for $K<0$.

\begin{enumerate}

\item Let a saddle atom be noncentral (an example of a cross section of such an atom before factorization by the involution~\eqref{Eq:CotangInvolution} is schematically shown in Fig.~\ref{Fig:SaddleNoStarNoInv}).  For an ``outer'', according to Fig.~\ref{Fig:SaddleNoStarNoInv},  Liouville torus $\mu_{\mathrm{out}} = -\alpha_r$, for``inner'' tori $\mu_{\mathrm{in}} = \alpha_r$.

\item Let the saddle atom be central and have no stars (an example of a cross section of such an atom before factorization by the involution~\eqref{Eq:CotangInvolution} is schematically shown in Fig.~\ref{Fig:SaddleNoStarInv}). In accordance with Fig.~\ref{Fig:SaddleNoStarInv}  there will be $3$ types of $\mu$-cycles corresponding to different Liouville boundary tori --- ``outer''  $\mu_{\mathrm{out}}$, ``inner noncentral'' $\mu_{\mathrm{in}}$ and ``inner central''  $\mu_{\mathrm{cent}}$.They can be taken as follows:  \[\mu_{\mathrm{out}} = \begin{cases} \frac{\alpha_{\varphi} -\alpha_r}{2}, \quad & K>0, \\ \frac{-\alpha_{\varphi} -\alpha_r}{2}, \quad & K<0,  \end{cases} \qquad \mu_{\mathrm{in}} = \alpha_r, \qquad \mu_{\mathrm{cent}} =  \begin{cases} \frac{\alpha_{r} -\alpha_{\varphi}}{2}, \quad & K>0, \\ \frac{\alpha_r+\alpha_{\varphi} }{2}, \quad & K<0.  \end{cases}  \]

\item 
Let a saddle atom be central and have stars  (an example of a cross section of such an atom before factorization by the involution~\eqref{Eq:CotangInvolution} is schematically shown in Fig.~\ref{Fig:SaddleStar}).  For an ``inner'', according to Fig. ~\ref{Fig:SaddleStar}, Liouville tori $\mu_{\mathrm{in}} = -\alpha_r$,  and for an ``outer'' torus $\mu_{\mathrm{out}} =  \frac{\alpha_{\varphi} -\alpha_{r}}{2}$ for $K>0$ and $\frac{-\alpha_r-\alpha_{\varphi} }{2}$ for $K<0$.

\end{enumerate}

\end{proposition}

 \begin{figure}[h]
 \begin{subfigure}{0,5\textwidth}
  \centering
 \includegraphics{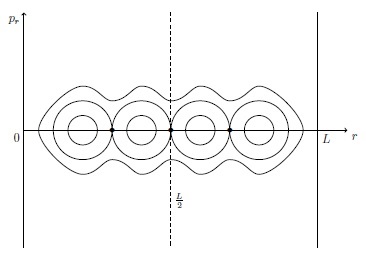}
\caption{Atom with a star}
\label{Fig:SaddleStar}
\end{subfigure}
~
  \begin{subfigure}{0,5\textwidth}
   \centering
   \includegraphics{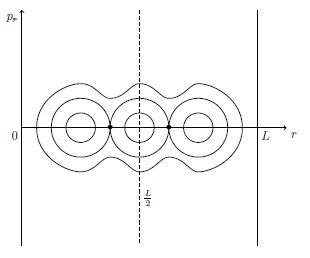}
\caption{Central atom without *}
\label{Fig:SaddleNoStarInv}
\end{subfigure}
\\
  \begin{subfigure}{\textwidth}
   \centering
   \includegraphics{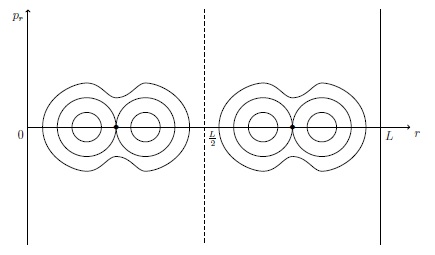}
\caption{Noncentral atom}
\label{Fig:SaddleNoStarNoInv}
\end{subfigure}
 \caption{Section of a saddle atom}
\label{ris:image1}
\end{figure}

\begin{proof} [Proof of Proposition~\ref{A:AdmBasSaddle}] Proposition~\ref{A:AdmBasSaddle} is proved similarly to Proposition~\ref{A:AdmBas}. Moreover, most of the calculations can be omitted, since in Proposition~\ref{A:AdmBas}  for each Liouville torus basis cycles  and their orientation were already found. Also, according to Remark~\ref{Rem_Syst_Sym}  it is enough to construct a basis for  $K>0$: in the coordinates $(p_{r'}, p_{\varphi'}, r', \varphi') = (p_r, - p_{\varphi}, r, - \varphi)$ for saddle atoms, as well as for elliptic atoms, the admissible basis $(\lambda', \mu')$ is given by the same formulas as the basis $(\lambda, \mu)$.

As an example, consider an “outer” torus for the case corresponding to Fig.~\ref{Fig:SaddleNoStarNoInv} for $K>0$. For a saddle atom $\lambda = \alpha_{\varphi}$ for $K>0$ since it is a cycle corresponding to a critical circle. For the corresponding atom  $A$  the admissible basis is $\lambda_A = \alpha_r, \mu_A = \alpha_{\varphi}$.  The orientation of the ``outer'' Liouville tori will be like that for the corresponding $A$ atoms, while the ``inner'' tori will have the opposite orientation. Therefore, the remaining cycle $\mu_{\mathrm{out}} = -\alpha_r$.

We note specific features for the choice of bases in the saddle case. For the case corresponding to Fig.~\ref{Fig:SaddleNoStarInv}, a section must pass through the cycles $\mu_i$. Therefore, we take opposite signs for $\alpha_{\varphi}$ in $\mu_{\mathrm{out}}$ and $\mu_{\mathrm{cent}}$ so that they give zero in the sum. For the case corresponding to Fig.  ~\ref{Fig:SaddleStar}, the $\mu$-cycles are in accordance with Remark ~\ref{Rem:StarAtom}.

Proposition~\ref{A:AdmBasSaddle} is proved. \end{proof}

\section{Proof of Theorem~\ref{Th:main1}} \label{S:ProofTh1}

\begin{proof}[Proof of Theorem~\ref{Th:main1}] We explicitly compute all gluing matrices in the admissible bases described in Section~\ref{S:AdmBasis}.  In all cases, the labels $r$ and $\varepsilon$ are calculated using the gluing matrices as described in \cite{BolFom}.

Denote by  $(\lambda_{+}, \mu_{+})$  and  $(\lambda_{-}, \mu_{-})$  the admissible bases for $K>0$ and $K<0$ respectively. In all cases, the cycle $\alpha_{\varphi}$ is correctly defined for all  $K$.  For the cycles $\alpha_r$ this is not the case because the coordinates $(r, \varphi)$ have singularities at the poles. Therefore, if the level  $K=0$ contains a pole, then the cycles $\alpha_{r\pm}$ for $K>0$ and $K<0$ can go into different cycles after passing to the level  $K=0$.

\begin{enumerate}

\item First, consider the case \ref{Item:RP2S1S3Case}. Since for $K=0$  the Liouville torus does not contain poles, $(p_r, p_{\varphi}, r, \varphi)$ will be global coordinates on $Q^3$. Therefore, the formulas for admissible bases from Proposition ~\ref{A:AdmBas} will be satisfied for all  $K$. The gluing matrix for noncentral atoms has the form \[ \left( \begin{matrix} \lambda_{+} \\ \mu_{+} \end{matrix} \right) = \left( \begin{matrix} 1 & 0 \\ 0 & -1 \end{matrix} \right) \left(\begin{matrix} \lambda_{-} \\ \mu_{-} \end{matrix} \right), \]   and for central atoms has the form \[ \left( \begin{matrix} \lambda_{+} \\ \mu_{+} \end{matrix} \right) = \left( \begin{matrix} 1 & 0 \\ 1 & -1 \end{matrix} \right) \left(\begin{matrix} \lambda_{-} \\ \mu_{-} \end{matrix} \right). \]

\item Now consider the case ~\ref{Item:RP2SphereCase}.  It is easy to show that the following relation holds \begin{equation} \label{Eq:BasisThroughPole} \alpha_{r+} = \alpha_{r-} - \alpha_{\varphi}. \end{equation} Fig.~\ref{Fig:LambdaChange} schematically shows how the cycle $\alpha_{r+}$ (continuously) changes with a variation of $K$. From~\eqref{Eq:BasisThroughPole} and Proposition~\ref{A:AdmBas} it follows that in this case the gluing matrix has the form  \begin{equation} \label{Eq:GlueMatrixOnePole} \left( \begin{matrix} \lambda_{+} \\ \mu_{+} \end{matrix} \right) = \left( \begin{matrix} 1 & 1 \\ 0 & -1 \end{matrix} \right) \left(\begin{matrix} \lambda_{-} \\ \mu_{-} \end{matrix} \right). \end{equation}

\begin{figure}
  \centering
\includegraphics{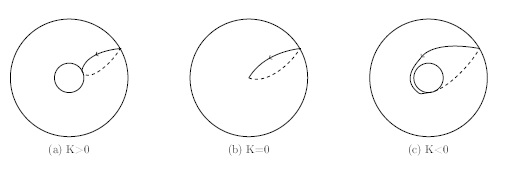}
  \caption{Bifurcation of the $\lambda$ in case~\ref{Item:RP2SphereCase}}
  \label{Fig:LambdaChange}
\end{figure}

Let us demonstrate this with a simple example: consider a  natural  system on the plane $\mathbb{R}^2(x, y)$ with Hamiltonian $H = \frac{p_x^2 + p_y^2 }{2} + \frac{x^2 +y^2}{2}$ and first integral $K = x p_y - y p_x$.  It is easy to see that this system has one singular point of rank $0$ of center-center type. In the coordinates \[ u_1 = \frac{p_x + y}{\sqrt{2}}, \quad v_1 = \frac{p_y -x}{\sqrt{2}}, \quad u_2 = \frac{p_x - y}{\sqrt{2}}, \quad v_2 = \frac{p_y +x}{\sqrt{2}},\] the form is  $\omega = du_1 \wedge dv_1 + du_2\wedge dv_2$, and the integrals $F_{1} = \frac{H-K}{2} = \frac{u_1^2 +v^2_1}{2}$ and $F_{2} = \frac{H+K}{2} = \frac{u_2^2 +v^2_2}{2}$ have a canonical form, as in the Eliasson theorem (see \cite{BolFom}). Denote by $\alpha_{1}$ and $\alpha_{2}$ the cycles corresponding to the trajectories of the Hamiltonian vector fields $\operatorname{sgrad} F_1$ and $\operatorname{sgrad} F_2$ respectively. Similarly to the proof of Proposition~\ref{A:AdmBas} it is easy to show that we  can take $\lambda_{+} = \alpha_{1}, \mu_{+}' = \alpha_2$  as admissible bases (in a neighbourhood of points $F_1 =0$) and $\lambda_{-} = \alpha_{2}, \mu_{-}' = \alpha_1$ (in a neighbourhood of points $F_2 =0$). For systems on $\mathbb{RP}^2$ the cycles $\mu_{\pm}$ correspond to the flow of$\operatorname{sgrad} K$. Therefore we have to take $\mu_{+} = \alpha_2 - \alpha_1$ and $\mu_{-} = \alpha_1 - \alpha_2$. Thus, we obtain the required relation~\eqref{Eq:GlueMatrixOnePole}.

\item Finally, consider the case~\ref{Item:RP2LensCase}.  Similarly to the case ~\ref{Item:RP2SphereCase}  (see also \cite[Volume 2, Section  3.4]{BolFom}) we can show that  \begin{equation} \label{Eq:BasisThroughPole2} \alpha_{r+} = \alpha_{r-}  - 2 \alpha_{\varphi}. \end{equation} From~\eqref{Eq:BasisThroughPole2} and Proposition~\ref{A:AdmBas}  it follows that in this case the gluing matrix has the form \[ \left( \begin{matrix} \lambda_{+} \\ \mu_{+} \end{matrix} \right) = \left( \begin{matrix} 3 & 4 \\ 1 & 1 \end{matrix} \right) \left(\begin{matrix} \lambda_{-} \\ \mu_{-} \end{matrix} \right). \]

\end{enumerate}

Theorem~\ref{Th:main1} is completely proved. \end{proof}

\section{Proof of Theorem~\ref{Th:main2}}

\begin{proof}[Proof of Theorem~\ref{Th:main2}]  Theorem~\ref{Th:main2} is proved similarly to Theorem~\ref{Th:main1}. We only list all gluing matrices between bifurcations. By $\lambda^A_{\pm}, \mu^A_{\pm}$ we will denote admissible bases for an atom   $A$, and by $\lambda^V_{\pm}, \mu^V_{\pm}$ we will denote admissible bases for saddle atoms (both with stars and without). The sign $\pm$ is positive for the atom corresponding to a larger value of  $K$  (and negative for a smaller value of  $K$).

Table~\ref{Tab:EllSaddleMatr} shows the gluing matrix $C$ for edges of the form  $A-V$ and $V-A$ (that is, between elliptic and saddle atoms). Here we assume that $\left( \begin{matrix} \lambda^A_{+} \\ \mu^A_{+} \end{matrix} \right) = C \left(\begin{matrix} \lambda^V_{-} \\ \mu^V_{-} \end{matrix} \right) \displaystyle$ for $K>0$ and $\left( \begin{matrix} \lambda^V_{+} \\ \mu^V_{+} \end{matrix} \right) = C \left(\begin{matrix} \lambda^A_{-} \\ \mu^A_{-} \end{matrix} \right) \displaystyle$ for $K<0$.

\begin{table}[h!]
\centering

{
\renewcommand{\arraystretch}{1.5}%

\begin{tabular} {c | c| c }
& Noncentral $A$   & Central $A$ \\ \hline
$K>0$ &  $\left( \begin{matrix} 0 & 1 \\ 1 & 0 \end{matrix} \right)$ &$\left( \begin{matrix} 1 & 2 \\ 1 & 1 \end{matrix} \right)$  \\ \hline
$K<0$ &  $\left( \begin{matrix}0 & 1 \\ 1 & 0 \end{matrix} \right)$, & $\left( \begin{matrix} -1 & 2 \\ 1 & -1 \end{matrix} \right)$  \\
\end{tabular}

}

\caption{Gluing matrices between elliptic and saddle atoms}
\label{Tab:EllSaddleMatr}

\end{table}

For noncentral edges of the form $V-V$  (that is between saddle atoms) \[ \left( \begin{matrix} \lambda^V_{+} \\ \mu^V_{+} \end{matrix} \right) = \left( \begin{matrix} 1 & 0 \\ 0 & -1 \end{matrix} \right) \left(\begin{matrix} \lambda^V_{-} \\ \mu^V_{-} \end{matrix} \right). \]

The gluing matrix $C$ on the central edge of the form $V-V$  depends on the type of projection  $\pi(Q^3)$ on $\mathbb{RP}^2$ is shown in Table~\ref{Tab:SaddleSaddleCent}. We assume that $\left( \begin{matrix} \lambda^V_{+} \\ \mu^V_{+} \end{matrix} \right) = C \left(\begin{matrix} \lambda^V_{-} \\ \mu^V_{-} \end{matrix} \right) \displaystyle$. 
It suffices to take an advantage of the fact that we know the expression of admissible bases in terms of $\alpha_{r}$ and $\alpha_{\varphi}$ (see Proposition~\ref{A:AdmBasSaddle}) and how these cycles change with a change of the sign of $K$ (see the proof of Theorem~\ref{Th:main1}). Namely, in all cases, the cycle $\alpha_{\varphi}$ does not change, $\alpha_{r+} = \alpha_{r-}$ if $\pi(Q^3) \approx I^1 \times S^1$ or $\mathbb{M}^2$,  \eqref{Eq:BasisThroughPole} holds if $\pi(Q^3) \approx D^2$, and \eqref{Eq:BasisThroughPole2} holds if $\pi(Q^3) = \mathbb{RP}^2$   .

\begin{table}[h!]
\centering

{
\renewcommand{\arraystretch}{1.5}%

\begin{tabular} {c | c| c | c |c }
$\pi(Q^3)$ & $\mathbb{RP}^2$ & $D^2$ & $\mathbb{M}^2$ & $I^1 \times S^1$    \\ \hline
$C$ &  $\left( \begin{matrix}-1 & 0 \\-2 & 1 \end{matrix} \right)$ &  $\left( \begin{matrix}-1 & 0 \\-1 & 1 \end{matrix} \right)$ &  $\left( \begin{matrix}-1 & 0 \\-1 & 1 \end{matrix} \right)$   & $\left( \begin{matrix}-1 & 0 \\0 & 1 \end{matrix} \right)$  
\end{tabular}

}

\caption{Gluing matrices on the central edge between saddle atoms}
\label{Tab:SaddleSaddleCent}

\end{table}

All the labels  ($r, \varepsilon$ and $n$) are calculated using gluing matrices using the formulas from  \cite{BolFom}. Theorem~\ref{Th:main2} is completely proved. \end{proof}

\begin{remark} \label{Rem:Topalov} We demonstrate that  according to the Topalov formula (see  \cite[Volume 2, Section 1.9]{BolFom}) for $V(r)=0$ the label $n$ is equal to $0$ or $-2$,  which is consistent with Theorem~\ref{Th:main2},  but does not agree with the label $n$ specified in \cite[Volume 2, Theorem  3.11]{BolFom}.  If $V(r)=0$, then for any component of a nonsingular isoenergy surface $\pi(Q^3) = \mathbb{RP}^2$ and $Q^3 \approx L_{4, 1}$. Hence $\operatorname{Tor} H_1( Q^3) = 4$. There are two possible variants.

\begin{enumerate}

\item Let $r = \frac{L}{2}$ be a local maximum of the function $U_h(r)$. In this case, there are no atoms with stars in the molecule, but there are two central atoms $A$ (an example of such an atom and a function  $U_h(r)$  are shown in Fig.~\ref{Fig:MoleculeNoCenter}).

  \begin{figure}[h!]
   \centering
   \includegraphics{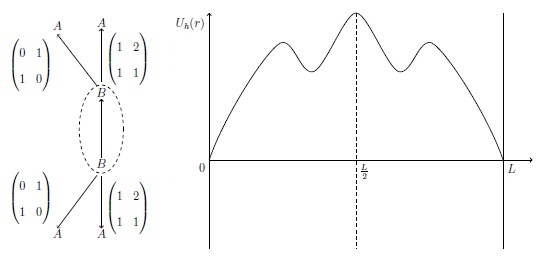}
 \caption{Molecule without central atoms $A$}
\label{Fig:MoleculeNoCenter}
\end{figure}

On the corresponding two edges coming from a single family, the gluing matrix is $\left( \begin{matrix} 1 & 2 \\1 & 1 \end{matrix} \right) \displaystyle$ and the label $r = \frac{1}{2} \displaystyle$  (note that for  $K<0$  we reversed an edge, and therefore we need to take not the matrix from Table~\ref{Tab:EllSaddleMatr} but the inverse of it).  On the remaining outgoing edges, the gluing matrix is  $\left( \begin{matrix} 0 & 1 \\1 & 0 \end{matrix} \right) \displaystyle$ and the label $r = 0$.  In the notation from  \cite{BolFom} the energy of the equipped molecule is  $N(W^*) =4$ and \begin{equation} \label{Eq:TopalovMax} N(W^*) = \beta_1 \dots \beta_m \tilde{n}.\end{equation} Two denominators of $r$-labels are equal to two: $\beta_1 = \beta_2 = 2$ and the rest $\beta_i =0$. Therefore, the ``energy of the family'' is $\tilde{n} = \pm 1$. On the other hand, $\tilde{n} = n  + \sum\limits_{\textrm{outer edges}}^{} r_i  = n +1$. Thus $n = 0$ or $-2$. 

\item Let $r = \frac{L}{2}$ be the local minimum of the function $U_h(r)$.  In this case, in the molecule  there are $p=2$ atoms with a star but there are no central atoms $A$ (an example of such an atom and a function $U_h(r)$ are shown in Fig.~\ref{Fig:Molecule2Center}).

 \begin{figure}[h!]
  \centering
 \includegraphics{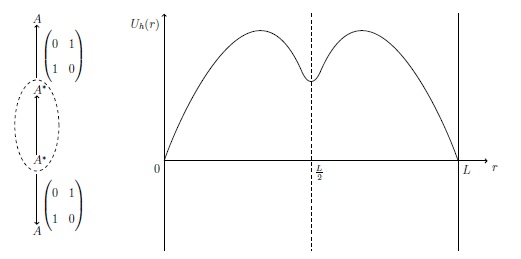}
 \caption{Molecule with central atoms $A$}
\label{Fig:Molecule2Center}
\end{figure}

On all outgoing edges, the transition matrix is  $\left( \begin{matrix} 0 & 1 \\1 & 0 \end{matrix} \right) \displaystyle$
and the label $r = 0$.  According to the rules from \cite{BolFom} the formula~\eqref{Eq:TopalovMax} will be replaced by  \[ N(W^*) = 4 \beta_1 \dots \beta_m \tilde{n},\] and the ``energy of the family'' $\tilde{n} = n   + \sum\limits_{\textrm{outer edges}}^{} r_i  +\frac{p}{2} = n +1$.  From where, again, $n = 0$ or $-2$.

\end{enumerate}

The ambiguity in the label  $n$ is due to a choice of orientation of $Q^3$. If we change the orientation of $Q^3$, then the label $n=0$ changes to $n=-2$ and vice versa (see  \cite[Volume 1, Section 4.5.2]{BolFom}). We have chosen the orientation so that $n=-2$. \end{remark}

\section*{Acknowledgments}

I.\,K.~Kozlov would like to thank Prof.~A.\,V.~Bolsinov for useful comments and tips when writing the paper. The
work of I.\,K.~Kozlov was supported by the Russian Science Foundation grant (project № 17-11-01303).

 \end{document}